\documentclass[graybox]{svmult}

\usepackage{mathptmx}       
\usepackage{helvet}         
\usepackage{courier}        
\usepackage{type1cm}        
\usepackage{makeidx}         
\usepackage{graphicx}        
\usepackage{multicol}        
\usepackage[bottom]{footmisc}

\usepackage{amsmath}
\usepackage{amsfonts}

\newcommand{\EndProof}{\hfill $\square$ \\ \indent}

\usepackage{dsfont}
\newcommand{\param}{ \sigma, r, R}

\newcommand{\rR}{r, R}
\newcommand{\md}{\mathrm{d}}
\newcommand{\g}{\zeta}
\newcommand{\dstar}{ \md \sigma \, \md v_* \md I_*}
\newcommand{\dall}{ \md \sigma \,   \md v_* \md I_* \, \md v \, \md I}
\newcommand{\dii}{   \md v_* \md I_* \, \md v \, \md I}
\newcommand{\ipar}{ \mathbb{S}^2 }
\newcommand{\ii}{(\mathbb{R}^3 \times \mathbb{R}_+)^2}
\newcommand{\iall}{(\mathbb{R}^3 \times \mathbb{R}_+)^2 \times  \mathbb{S}^2 }
\newcommand{\ivI}{\mathbb{R}^3 \times \mathbb{R}_+}
\newcommand{\la}{\langle}
\newcommand{\ra}{\rangle}

\newcommand{\As}{\tilde{A}_k}

\newcommand{\Qf}{Q^{\text{\textup{f}}}}
\newcommand{\Qff}{Q_{\zeta_{\text{\textup{f}}}}^{\text{\textup{f}}}}
\newcommand{\Qp}{Q_\zeta}
\newcommand{\gf}{\zeta_{\text{\textup{f}}}}
\newcommand{\Qom}{Q^{\omega}}
\newcommand{\Bf}{B^{\text{{\textup{f}}}}}
\newcommand{\Bft}{\tilde{B}^{\text{{\textup{f}}}}}

\newcommand{\mv}{\mathfrak{m}^v}
\newcommand{\mI}{\mathfrak{m}^I}
\newcommand{\mvI}{\mathfrak{m}}

\usepackage[dvipsnames]{xcolor}

\makeindex             

\begin{document}

\title*{Moment estimates for polyatomic Boltzmann equation with frozen collisions}
\author{Ricardo Alonso and Milana \v Coli\'c}
\institute{Ricardo Alonso \at Division of Arts \& Sciences, Texas A\&M  University at Qatar, Education City, Doha, Qatar, \email{ricardo.alonso@qatar.tamu.edu}
\and Milana \v Coli\'c \at Department of Mathematics and Informatics, Faculty of Sciences, University of Novi Sad, Trg Dositeja Obradovi\'ca 4, 21 000 Novi Sad, Serbia \at and Department of Mathematics and Scientific Computing, University of Graz,  Heinrichstra\ss e 36, 8010 Graz, Austria,
\email{milana.colic@dmi.uns.ac.rs}}
%
%
\maketitle

\abstract{In this paper, a polyatomic gas with continuous internal energy is considered, allowing for frozen collisions, in which the kinetic energy of the colliding particle pair is conserved, and the internal energy of each particle remains unchanged. \emph{A priori} moment estimates are derived for solutions of the space-homogeneous Boltzmann equation with a collision kernel of the hard potentials type with cut-off. The model with frozen collisions is first analyzed, followed by a review of general collisions—referred to as pure polyatomic—which preserve the total kinetic and internal energy. By combining existing results for pure polyatomic collisions with the newly derived estimates for frozen collisions, moment estimates are established for the Boltzmann equation with a collision operator that convexly combines both types of collisions. In particular, the moment generation property is shown to be driven by the rate of the pure polyatomic operator, and the moment propagation property holds.}


\section{Introduction}

This paper is concerned with the Boltzmann equation modelling a polyatomic gas within the continuous approach  for internal energy. In such a setting, gas is described by a distribution function $f(t,x,v,I)\geq0$ that depends on time $t>0$, space position $x\in\mathbb{R}^3$,   molecular velocity $v\in\mathbb{R}^3$ and internal energy $I\geq0$. Evolution of $f$, as governed by the Boltzmann equation,   accounts for the  effects of  collisions on gas dynamics through the collision operator.  The base assumptions used to model collisions profoundly  influence   the form of the collision operator and subsequent analysis. One  such assumption is that collisions are binary and conserve  both momentum and total (kinetic + internal) energy. Namely, for  two molecules of mass $m$, with velocity-internal energy pairs $(v',I')$ and $(v'_*,I'_*)$, which collide and give rise to molecules of the same mass $m$ with $(v,I)$ and $(v_*,I_*)$, respectively, the following rules apply
 \begin{equation}\label{coll CL poly}
 	v' + v'_*  = v + v_*, \qquad \frac{m}{2} |v'|^2 + I' +  \frac{m}{2} |v'_*|^2 + I'_* = \frac{m}{2} |v|^2 + I +  \frac{m}{2} |v_*|^2 + I_*.
 \end{equation}
Within this framework, the  collision operator and the Boltzmann equation have been studied in the space homogeneous case in \cite{Gamba-Colic-poly,Alonso-Gamba-Colic,MC-Alonso-Lp, MC-Alonso-Pesaro}, and in perturbative setting in \cite{Bern,Brull,Duan-Li,Ko-Son}, all assuming a collision kernel form corresponding to hard potentials in both relative velocity and internal energy. 

The question of physical relevance of the  aforementioned collision kernel  by modelling of transport coefficients (shear and bulk viscosity and heat conductivity), based on  evaluation of the full collision operator, was raised in \cite{Djordj-Colic-Spa,Djordj-Colic-Torr,Djordj-Obl-Colic-Torr}. For recovering a correct value of the Prandtl number, it was crucial to introduce in the  description  collisions that preserve only the kinetic energy. It has been already known in literature dealing with operators of relaxation type, such as BGK, that various exchange processes occur on different characteristic time scales. The fact that translational energy is exchanged between particles  in all collisions and the internal energy is exchanged only in some of the collisions is accounted by introducing a two term operator where the first term only models the translational exchange processes and the second term models translational and internal exchange processes \cite{Frozen Str,Frozen Wu}. The concept of frozen (elastic) collisions is also used in direct simulation Monte Carlo (DSMC) method where a probability of  internal energy relaxation event is prescribed for two particles selected for a collision \cite{Frozen DSMC}.

These considerations motivate us to incorporate frozen collisions in the polyatomic Boltzmann equation and study moment estimates for a solution of such equation in the space homogeneous case. Namely, frozen collisions refer to collisions in which the internal energy of each molecule remains invariant \cite{Djordj-Colic-Torr}, i.e.
\begin{equation}\label{CL frozen}
v'+v'_* = v+v_*, \quad  |v'|^2 +    |v'_*|^2 = |v|^2 +     |v_*|^2, \quad   I' = I,  \quad  I'_* =  I_*.
\end{equation}
The Boltzmann equation we consider in this paper convexly combines, with a factor $\omega\in[0,1]$,  the collision operator $\Qp(f,f)(v,I)$ for pure polyatomic (or non-frozen) collisions \eqref{coll CL poly}  with  potential rate $\g$ and the collision operator $\Qff(f,f)(v,I)$ for frozen collisions \eqref{CL frozen} with potential rate $\gf$, 
\begin{equation}\label{omega coll op BE}
	\partial_t f(t,v,I) = \Qom(f,f)(v,I) := \omega \ \Qp(f,f)(v,I) + (1-\omega) \	\Qff(f,f)(v,I).
\end{equation}
Obviously, the case $\omega=1$ corresponds to results already established in \cite{Gamba-Colic-poly,Alonso-Gamba-Colic} recently reviewed for the case of a single component gas in \cite{MC-Alonso-Pesaro}. Here we will consider the frozen case ($\omega=0$) and combine the results to conclude on $\omega\in(0,1)$.

The paper is organized as follows. In Section \ref{Sec: Frozen}, the frozen collision operator is introduced and moment estimates for the Boltzmann equation \eqref{omega coll op BE} with $\omega=0$ are studied in Section \ref{Sec: ME frozen}. In particular, for solutions with finite mass and energy, we show the propagation property of moments of any order $k>2$. Then, Section \ref{Sec: Convex} considers the Boltzmann equation  \eqref{omega coll op BE} with $\omega \in (0,1)$ and  proves moment generation estimate governed by the hard potentials rate  $\g$ of pure polyatomic collision operator and moment propagation estimate for any $k$-moment, $k>2$, of solutions to  \eqref{omega coll op BE}  with finite mass and energy. Notation of moments is introduced below.

\subsection*{Notation}
To be consistent with \cite{Alonso-Gamba-Colic}, we define the brackets,
\begin{equation}\label{brackets}
	\la v \ra = \sqrt{1 + \tfrac{1}{2} |v|^2 }, \qquad \la I \ra = \sqrt{  1  + \tfrac{1}{m} I}, \qquad \la v, I \ra  = \sqrt{  1+\tfrac{1}{2} |v|^2 + \tfrac{1}{m} I}.
\end{equation}
Then, polynomial moment of order $k\geq0$ associated to any suitable function $f$ is defined with respect to brackets \eqref{brackets} as follows,
\begin{equation}
	\label{moments}
	\begin{split}
		&\mathfrak{m}^v_k[f] = 	\int_{\ivI} f(v,I) \la v \ra^k \,  \md v \,  \md I, \qquad 	\mathfrak{m}^I_k[f] = 	\int_{\ivI} f(v,I) \la I \ra^k \,  \md v \,  \md I,\\
		&\mathfrak{m}_k[f] = 	\int_{\ivI} f(v,I) \la v, I \ra^k \,  \md v \,  \md I.
	\end{split}
\end{equation}
Particular to this paper are partial moments $\mathfrak{m}^v_k$ and $\mathfrak{m}^I_k$, important for the frozen model, defined with respect to brackets in $v$ and $I$, to which we refer as $v-$ and $I-$moments, respectively. 

\section{Frozen collision operator}\label{Sec: Frozen}

Similar to the monatomic case, the  frozen collisions \eqref{CL frozen} are parametrized with an angular variable $\sigma\in\mathbb{S}^2$ to express 
\begin{equation}\label{coll rules frozen}
	v'  = \frac{v+v_*}{2} + \frac{|v-v_*|}{2} \sigma,  \qquad
	v'_{*} =  \frac{v+v_*}{2}  -   \frac{|v-v_*|}{2} \sigma, \quad \sigma \in \mathbb{S}^2.
\end{equation}
The corresponding frozen collision operator is defined by \cite{Djordj-Colic-Torr}
\begin{multline}\label{coll operator frozen}
\Qf(f,g)(v,I) = \int_{\mathbb{R}^3\times\mathbb{R}_+\times\mathbb{S}^2}  \left\{ f(v',I)g(v'_*,I_*)   - f(v,I)g(v_*, I_*) \right\} \\ \times  \Bf(v,v_*,I,I_*,\sigma)  \,\md \sigma \, \md v_* \, \md I_*,
\end{multline}
where the collision kernel $\Bf(v,v_*,I,I_*,\sigma)\geq 0$ satisfies micro-reversibility 
\begin{equation*}
	\Bf(v,v_*,I,I_*,\sigma)  =  \Bf(v',v'_*,I,I_*,\hat{u})  = \Bf(v_*,v,I_*,I,-\sigma). 
\end{equation*}
Since the Jacobian of the transformation $(v,v_*,I,I_*,\sigma)\mapsto(v',v'_*,I,I_*,\sigma')$, with \eqref{coll rules frozen} and $\sigma'=\tfrac{v-v_*}{|v-v_*|}$, is one, the weak form of the collision operator \eqref{coll operator frozen} can be defined, for a suitable test function $\chi(v,I)$,
\begin{multline}\label{coll operator frozen weak}
\int_{\mathbb{R}^3\times\mathbb{R}_+}  \left\{	\Qf(f,g)(v,I) + 	\Qf(g,f)(v,I) \right\} \chi(v,I) \, \md v\, \md I   
	\\	=   \int  \left\{  \chi(v', I) + \chi(v'_*, I_*) - \chi(v, I)- \chi(v_*, I_*) \right\} 
	\\ \times  f(v, I) \, g(v_*, I_*) \, \Bf(v,v_*,I,I_*,\sigma)  \, \dall.
\end{multline}
The collision invariants are the natural ones determined by the collision rules \eqref{coll rules frozen}, i.e. $1$, $v$, $ |v|^2$, $I$, but also any function $\chi(I)$ such that \eqref{coll operator frozen weak} makes sense,
\begin{equation}\label{coll invariants}
	\int_{\mathbb{R}^3\times\mathbb{R}_+}  \left\{	\Qf(f,g)(v,I) + 	\Qf(g,f)(v,I) \right\} 
 \begin{pmatrix}
		1 \\ v \\  |v|^2 \\ \chi(I)
	\end{pmatrix}   \, \md v\, \md I    =0.
\end{equation}
For the collision operator \eqref{coll operator frozen}, the space homogeneous Boltzmann equation reads 
\begin{equation}\label{BE frozen}
	\partial_t f(t,v,I) = \Qf(f,f)(v,I),
\end{equation}
and corresponds to \eqref{omega coll op BE} with $\omega=0$. For the moment, we skip  notation emphasizing $\gf$ in the superscript of the collision operator until the  final Section \ref{Sec: Convex}, since Sections \ref{Sec: Frozen}, \ref{Sec: ME frozen} deal only with the frozen collision operator.

\subsection{Assumption on the collision kernel}\label{Sec: assumpt coll kernel}
The collision kernel is assumed to take a factorized form
\begin{equation}\label{assumpt B factor}
\Bf(v,v_*,I,I_*,\sigma)   =  b(\hat{u} \cdot\sigma) \ \Bft(v,v_*,I,I_*),
\end{equation}
with $b$ non-negative and integrable  $b\in L^1(\mathbb{S}^2;\md \sigma)$, and $\Bft$ taking the form of hard potentials, up to a constant depending on $\g$,
\begin{equation}\label{assump-tilde-B}
	c_\g \,   \left( \frac{E}{m} \right)^{\g/2} \leq  \Bft(v,v_*,I,I_*) \leq    C_\g    \left( \frac{E}{m} \right)^{\g/2}, \quad \frac{E}{m}= \frac{1}{4}|u|^2+\frac{I+I_*}{m}, \quad  \g \in (0,2],
\end{equation}
In the sequel, we consider exclusively the collision operator  \eqref{coll operator frozen} describing frozen collisions with the collision kernel satisfying \eqref{assumpt B factor}--\eqref{assump-tilde-B}.

\section{Moment estimates for frozen collisions}\label{Sec: ME frozen}

Due to  invariance of the internal energy in  frozen collisions \eqref{CL frozen}, for any suitable function $\chi(I)$  depending only on the internal energy, we noticed previously in \eqref{coll invariants} that 
 the weak form \eqref{coll operator frozen weak} vanishes,
\begin{equation}\label{coll operator frozen I}
	\int_{\ivI}  \left\{	\Qf(f,g)(v,I) + 	\Qf(g,f)(v,I) \right\} \chi(I) \, \md v\, \md I  =0.
\end{equation}
For a solution of the Boltzmann equation \eqref{BE frozen}, this implies
\begin{equation*}
 \int_{\ivI} f(t,v,I) \, \chi(I) \md v \, \md I =  \int_{\ivI} f(0,v,I) \, \chi(I) \md v \, \md I.
\end{equation*}
In particular, choosing $\chi(I) = \la I \ra^k$, the following Proposition holds.

\begin{proposition}[Polynomial $I-$moment propagation estimate] If  $f$  is a solution of the Boltzmann equation \eqref{BE frozen} with finite $\mvI_2$-moment, then for $k\geq0$,
	\begin{equation}\label{poly prop I}
		\mI_{k}[f](t) = \mI_{k}[f](0), \qquad \forall t\geq 0.
	\end{equation}
\end{proposition}
In particular, if $ \mI_{k}[f](0)<\infty$ then $\mI_{k}[f](t)<\infty$ for any $t\geq0$, recovering the   classical propagation property of $I-$moments. Whereas, if $ \mI_{k}[f]$ is infinite at $t=0$, it will remain infinite, i.e. there will be no later creation of such moments.

Next, we  consider moments in the velocity variable. Since the collision rules \eqref{coll rules frozen} for the velocity variable are the same as for the monatomic case, the classical Povzner lemma applies \cite{Bob-Gamba-Panferov,Alonso-Lods}. 
\begin{lemma}[Povzner $\sigma$-averaging]\label{Lemma Povzner} Let $b(\hat{u}\cdot\sigma) \in L^1(\mathbb{S}^2;\md \sigma)$.
There exists an explicit positive constant $\mathcal{C}_k< \| b\|_{L^1}$ decreasing to zero in $k>2$, such that for $v'$ and $v'_*$ from \eqref{coll rules frozen} the following averaging  holds,
\begin{equation*}
\int_{\ipar} \left( \la v' \ra^k + \la v'_* \ra^k \right) b(\hat{u}\cdot\sigma) \, \md \sigma \leq \mathcal{C}_k  \left( \la v \ra^2 + \la v_* \ra^2\right)^{k/2}, \quad k>2.
\end{equation*}
\end{lemma}
This averaging property allows to  conclude on moments for the collision operator.
\begin{proposition}[Polynomial moment estimates  on the frozen collision operator]\label{Prop weak} Let $\g \in (0,2]$ and $k>2$.  For 	 suitable $f,g$ having finite $\mvI_2$ moments, the following estimate on $v-$moments of the frozen collision operator \eqref{coll operator frozen} holds
		\begin{multline}\label{v moment final}
	 		\int_{\ivI}  \left\{	\Qf(f,g)(v,I) + 	\Qf(g,f)(v,I) \right\} \la v \ra^k \, \md v\, \md I   
		\\	\leq  	 - \left( A_{k}[f,g]  (\mv_{k}[f])^{1+\frac{\g}{k-2}}  +  A_{k}[g,f]  (\mv_{k}[g])^{1+\frac{\g}{k-2}}  \right)+ B_k[f,g] + B_k[g,f], 
	\end{multline}
with positive constants $A_{k}>0$ and $B_k$ depending on $\mvI_2$ moments of $f$ and $g$, which are explicitly computed along the proof. 

Moreover, for $\g \in [0,2]$ and $k>2$, the following estimate on moments of the frozen collision operator \eqref{coll operator frozen} holds
		\begin{multline}\label{vI moment}
	\int_{\ivI}  \left\{	\Qf(f,g)(v,I) + 	\Qf(g,f)(v,I) \right\} \la v, I \ra^k \, \md v\, \md I   
	\\	\leq  	    D_k  \left( \mvI_2[f] \mvI_{k}[g] + \mvI_2[g] \mvI_{k}[f] \right),
\end{multline}
where $ D_k  =  2^{ {k}/2+2} \, C_\g  \| b\|_{L^1}$.
\end{proposition}
\begin{proof}
The weak form \eqref{coll operator frozen weak} for the test function $\chi = \la v \ra^k$ implies	
	\begin{multline*}
W :=		\int_{\ivI}  \left\{	\Qf(f,g)(v,I) + 	\Qf(g,f)(v,I) \right\} \la v \ra^k \, \md v\, \md I   
		\\	\leq   \int_{\ii}  \left\{ \mathcal{C}_k  \left( \la v \ra^2 + \la v_* \ra^2\right)^{k/2}- \|b\|_{L^1} \left(\la v\ra^k + \la v_* \ra^k \right) \right\} 
		\\ \times  f(v, I) \, g(v_*, I_*) \, \Bft(v,v_*,I,I_*)  \, \dall.
	\end{multline*}
Now, applying  p-Binomial inequality, see Lemma 5.5 from \cite{Alonso-Gamba-Colic},
\begin{multline}\label{pomocna 3}
	\left(  \la v \ra^{2} + \la v_* \ra^{2}  \right)^{k/2} \leq  \la v  \ra^{k}  +  \la v_* \ra^{k} 
	\\
	+  2^{ {k}/2+1} \left( \left\langle v \right\rangle^2 \left\langle v_* \right\rangle^{k- 2} \mathds{1}_{ \left\langle v \right\rangle \leq \left\langle v_* \right\rangle} + \left\langle v\right\rangle^{k- 2} \left\langle v_* \right\rangle^2  \mathds{1}_{   \left\langle v_* \right\rangle \leq \left\langle v  \right\rangle} \right),
\end{multline}
allows to split $W$ into positive  and negative   contributions,  
	\begin{equation}\label{W pomocna}
		\begin{split}
	W  	&\leq   \int_{\ii} f(v, I) \, g(v_*, I_*) \, \Bft(v,v_*,I,I_*) \Big\{ - \left( \|b\|_{L^1} -  \mathcal{C}_k   \right) \left(\la v\ra^k + \la v_* \ra^k \right) \\
	&	+   2^{ {k}/2+1} \,  \mathcal{C}_k  \left( \left\langle v \right\rangle^2 \left\langle v_* \right\rangle^{k- 2} \mathds{1}_{ \left\langle v \right\rangle \leq \left\langle v_* \right\rangle} + \left\langle v\right\rangle^{k- 2} \left\langle v_* \right\rangle^2  \mathds{1}_{   \left\langle v_* \right\rangle \leq \left\langle v  \right\rangle} \right)
	\Big\} 
  \, \dii. 
  \end{split}
\end{equation}
  First denote 
\begin{equation*}
 \tilde{A}_k =  \|b\|_{L^1} -  \mathcal{C}_k    >0, \quad k>2.
\end{equation*}
Then, the lower bound on the collision kernel  \eqref{assump-tilde-B} implies the estimate
\begin{equation}\label{Bf lower}
	\begin{split}
  \Bft(v,v_*,I,I_*) &\geq   	c_\g \,   \left( \frac{E}{m} \right)^{\g/2}  \geq   \tilde{c}_\g\,   \left(   \left(\frac{|u|^2}{4}\right)^{\g/2} + \left(\frac{I}{m}\right)^{\g/2}  + \left(\frac{I_*}{m}\right)^{\g/2}  \right) \\
   &\geq \tilde{c}_\g\,   \left( L \la v \ra^\g - \la v_* \ra^\g + \left(\frac{I}{m}\right)^{\g/2}  + \left(\frac{I_*}{m}\right)^{\g/2}  \right), 
   \end{split}
\end{equation}
with $  \tilde{c}_\g = c_\g\, 3^{\g/2-1},   L  = 2^{- \g} \min\left\{1, 2^{1-\g} \right\}$, see e.g. Lemma A.1 in  \cite{Alonso-Gamba-Colic}
, leading to 
\begin{multline*}
	\Bft(v,v_*,I,I_*) \left(\la v\ra^k + \la v_* \ra^k \right) 
	\\  \geq  \tilde{c}_\g  L \left(\la v\ra^{k+\g} +\la v_*\ra^{k+\g}  \right) - \tilde{c}_\g \left(   \la v \ra^k \la v_* \ra^\g +  \la v \ra^\g \la v_* \ra^k \right)  \\ + \tilde{c}_\g   \left( \la v\ra^k \left(\frac{I}{m}\right)^{\g/2}  + \la v_* \ra^k \left(\frac{I_*}{m}\right)^{\g/2}  \right).
\end{multline*}
On the other hand, the positive part of $W$ uses an upper bound, for $\g/2 \leq 1$, 
\begin{equation*}
	\left( \frac{E}{m} \right)^{\g/2} \leq \la v \ra^{\g} +  \la v_* \ra^{\g}  + \left(\frac{I}{m}\right)^{\g/2} + \left(\frac{I_*}{m}\right)^{\g/2}, 
\end{equation*}
which in combination with \eqref{assump-tilde-B} implies 
\begin{equation}\label{Bf upper}
	\begin{split}
	\Bft(v,v_*,&I,I_*)  \left( \left\langle v \right\rangle^2 \left\langle v_* \right\rangle^{k- 2} \mathds{1}_{ \left\langle v \right\rangle \leq \left\langle v_* \right\rangle} + \left\langle v\right\rangle^{k- 2} \left\langle v_* \right\rangle^2  \mathds{1}_{   \left\langle v_* \right\rangle \leq \left\langle v  \right\rangle} \right) 
	\\
	& \leq 2 \ C_\g  \left( \left\langle v \right\rangle^2 \left\langle v_* \right\rangle^{k- 2+\g}   + \left\langle v\right\rangle^{k- 2+\g} \left\langle v_* \right\rangle^2   \right) 
	\\
& \quad	+ C_\g\left(\tfrac{1}{m} \, I\right)^{\g/2} \left( \la v_*\ra^k +  \left( \varepsilon \left\langle v\right\rangle^k +  {\varepsilon^{-\frac{k-2}{2}}}  \right) \left\langle v_* \right\rangle^2      \right)
	\\
&	\quad	+ C_\g\left(\tfrac{1}{m} \, I_*\right)^{\g/2} \left( \la v\ra^k +  \left( \varepsilon \left\langle v_*\right\rangle^k +  {\varepsilon^{-\frac{k-2}{2}}}  \right) \left\langle v\right\rangle^2      \right),
\end{split}
\end{equation}
where we have conveniently  used the domain for velocity, and Young's inequality to get $  \left\langle v\right\rangle^{k- 2} \leq   \varepsilon \left\langle v\right\rangle^k +  {\varepsilon^{-\frac{k-2}{2}}}$. Gathering \eqref{Bf lower} and \eqref{Bf upper},  \eqref{W pomocna} becomes
	\begin{equation*}
	\begin{split}
	&	W  	\leq   \int_{\ii} f(v, I) \, g(v_*, I_*) \,  \\
		\times 
		\Big\{ & - \As \tilde{c}_\g  \left(   L \left(\la v\ra^{k+\g} +\la v_*\ra^{k+\g}  \right)    + m^{-\g/2} \left( \la v\ra^k \, {I}^{\g/2}  + \la v_* \ra^k \, {I_*}^{\g/2}  \right) \right) \\
	& + \As    \tilde{c}_\g \left(   \la v \ra^k \la v_* \ra^\g +  \la v \ra^\g \la v_* \ra^k \right)\\
		&	+   2^{ {k}/2+1} \,  \mathcal{C}_k \, C_\g  \Big\{ 2 \, \left( \left\langle v \right\rangle^2 \left\langle v_* \right\rangle^{k- 2+\g}   + \left\langle v\right\rangle^{k- 2+\g} \left\langle v_* \right\rangle^2   \right) 
		\\
		& \quad	+ m^{-\g/2} \left( \la v_*\ra^k \, I^{\g/2} +   \varepsilon \left\langle v\right\rangle^k  I^{\g/2}  \left\langle v_* \right\rangle^2   +  \varepsilon^{-\frac{k-2}{2}} I^{\g/2}  \left\langle v_* \right\rangle^2   \right)     
		\\
		&	\quad	+  m^{-\g/2} \left( \la v\ra^k \, I_*^{\g/2} +  \varepsilon \left\langle v_*\right\rangle^k  I_*^{\g/2} \left\langle v\right\rangle^2     +  \varepsilon^{-\frac{k-2}{2}} I_*^{\g/2}  \left\langle v\right\rangle^2      \right) \Big\}
		\Big\} 
		\, \dii.
	\end{split}
\end{equation*}
Choosing  an $\epsilon$ depending on $\mvI_0$ and $\mv_2$ moments,
\begin{equation*}
	\varepsilon[\cdot] =  \frac{\As \tilde{c}_\g      \mvI_0[\cdot] }{2^{ {k}/2+2} \,  \mathcal{C}_k \, C_\g    \mv_2[\cdot]},
\end{equation*}
the last inequality reads, in terms of moments defined in \eqref{moments},
	\begin{equation*}
	\begin{split}
			W  	\leq    & - \As \tilde{c}_\g     L \left(  \mvI_0[g] \, \mv_{k+\g}[f]    + \mvI_0[f] \,  \mv_{k+\g}[g]  \right)  \\
		&  - \frac{ \As \tilde{c}_\g  }{2}  m^{-\g/2} \bigg(  \mvI_0[g] \int_{\ivI} f(v, I) \la v\ra^k \, {I}^{\g/2} \md v\, \md I \\ & \qquad \qquad \qquad   +  \mvI_0[f] \int_{\ivI} g(v_*, I_*) \la v_* \ra^k \, {I_*}^{\g/2} \md v_*\, \md I_* \bigg)   \\
		& + \As    \tilde{c}_\g \left(  \mv_\g[g]  \, \mv_k[f]  +  \mv_\g[f] \, \mv_k[g]  \right)\\
		&	+   2^{ {k}/2+1} \,  \mathcal{C}_k \, C_\g  \Big\{ 2 \, \left( \mv_2[f] \, \mv_{k- 2+\g}[g]   +\mv_2[g] \, \mv_{k- 2+\g}[f]    \right) 
		\\
		& \quad	+  \mI_{\g}[f]   \,  \mv_k[g]  +  \varepsilon[g]^{-\frac{k-2}{2}} \, \mI_{\g}[f]  \mv_2[g]    
	 	+ \mI_{\g}[g]  \,  \mv_k[f]      +  \varepsilon[f]^{-\frac{k-2}{2}} \mI_{\g}[g]  \, \mv_2[f]       \Big\}.
	\end{split}
\end{equation*}
The terms can be regrouped by introducing the notation  
	\begin{equation*}
		\tilde{K}_1[\cdot] = \As    \tilde{c}_\g    \, \mv_\g[\cdot] + 2^{ {k}/2+1} \,  \mathcal{C}_k \, C_\g    \mI_{\g}[\cdot], \quad 	\tilde{K}_2[\cdot] = 2^{ {k}/2+2} \,  \mathcal{C}_k \, C_\g  \mv_2[\cdot],
	\end{equation*}
so that $W$ becomes, after neglecting the mixed moments,
	\begin{equation*}
		\begin{split}
			W  	\leq   
		 & - \As \tilde{c}_\g     L \left(  \mvI_0[g] \, \mv_{k+\g}[f]    + \mvI_0[f] \,  \mv_{k+\g}[g]  \right)  \\
			& +  \tilde{K}_1[g] \, \mv_k[f]   + \tilde{K}_1[f]\,   \mv_k[g]     	   + \tilde{K}_2[g] \, \mv_{k- 2+\g}[f]    +  \tilde{K}_2[f] \, \mv_{k- 2+\g}[g]
			\\
			&  	+ 2^{ {k}/2+1} \,  \mathcal{C}_k \, C_\g  \left(    \varepsilon[g]^{-\frac{k-2}{2}} \mI_{\g}[f]  \mv_2[g]    
			+       \varepsilon[f]^{-\frac{k-2}{2}} \mI_{\g}[g]  \mv_2[f]   \right),
		\end{split}
	\end{equation*}
Then, we follow the standard strategy of invoking moment interpolation formulas  together with Young's inequality  (see e.g. Eqs (110) and (113) in  \cite{Alonso-Gamba-Colic}),
\begin{align}
	&\mv_k[\cdot]  \leq (\mv_2[\cdot])^{\frac{\g}{k-2+\g}} (\mv_{k+\g}[\cdot])^{\frac{k-2}{k-2+\g}} \label{mi}
\\& \qquad\qquad \Rightarrow    \tilde{K}_1[g] \, \mv_k[f]  \leq K_1[f,g]  + \delta \, \mvI_0[g] \, \mv_{k+\g}[f], \nonumber \\
	&\mv_{k-2+\g}[\cdot]  \leq (\mvI_0[\cdot])^\frac{2}{k+\g} (\mv_{k+\g}[\cdot])^\frac{k-2+\g}{k+\g} \nonumber
\\& \qquad\qquad \Rightarrow  \tilde{K}_2[g] \, \mv_{k- 2+\g}[f]  \leq K_2[f,g] + \delta \,\mvI_0[g] \, \mv_{k+\g}[f], \nonumber
\end{align}
with constants
\begin{equation*}
	K_1[f,g] =\frac{(\tilde{K}_1[g])^{\frac{k-2+\g}{\g}} \mv_2[f]}{(\mvI_0[g] \delta)^{\frac{k-2}{\g}}}, \quad   K_2[f,g] =  \frac{(\tilde{K}_2[g])^{\frac{k+\g}{2}} \mvI_0[f]}{(\mvI_0[g] \delta)^{\frac{k-2+\g}{2}}}.
\end{equation*}
Then,  the choice  $\delta=\frac{\As \tilde{c}_\g     L}{4}$ allows an absorption into the negative part, and yields 
	\begin{equation*}
W	\leq   
		 -  \frac{\As \tilde{c}_\g     L}{2} \left(  \mvI_0[g] \, \mv_{k+\g}[f]    + \mvI_0[f] \,  \mv_{k+\g}[g]  \right) + B_k[f,g] + B_k[g,f],
\end{equation*}
with the constant
\begin{equation*}
	B_k[f,g] =  K_1[f,g]  
	+ K_2[f,g]  
	+ 2^{ {k}/2+1} \,  \mathcal{C}_k \, C_\g  
	\varepsilon[g]^{-\frac{k-2}{2}} \mI_{\g}[f]  \mv_2[g].
\end{equation*}
Then, invoking moment interpolation \eqref{mi}, 
	\begin{equation*}
	\begin{split}
		W  	\leq   
		& - \left( A_{k }[f,g]  (\mv_{k}[f])^{1+\frac{\g}{k-2}}  +  A_{k }[g,f]  (\mv_{k}[g])^{1+\frac{\g}{k-2}}  \right)+ B_k[f,g] + B_k[g,f],
	\end{split}
\end{equation*}
the final estimate follows, by introducing the notation
\begin{equation*}
A_{k}[f,g] =  \frac{\As \tilde{c}_\g     L}{2} \, \mvI_0[g] \,(\mv_2[f])^{-\frac{\g}{k-2}}.
\end{equation*}
For the part (b), we first rewrite the weak form of the collision operator
\begin{multline*}
	W :=		\int_{\ivI}  \left\{	\Qf(f,g)(v,I) + 	\Qf(g,f)(v,I) \right\} \la v, I \ra^k \, \md v\, \md I   
	\\	=   \int_{\iall}  \left\{ \la v',I \ra^k + \la v'_*,I_* \ra^k  -  \la v,I\ra^k + \la v_*,I_* \ra^k  \right\} 
	\\ \times  f(v, I) \, g(v_*, I_*) \, \Bf(v,v_*,I,I_*,\sigma)  \, \dall.
\end{multline*}
By the conservation law of the energy \eqref{CL frozen}, since $k>2$,
\begin{equation*}
	\la v', I \ra^{k} +  \la v'_*, I_* \ra_j^{k} \leq   \left(\la v', I \ra^{2} +  \la v'_*, I_* \ra^{2}  \right)^{k/2}  = \left(\la v, I \ra^{2} +  \la v_*, I_* \ra^{2}   \right)^{k/2}. 
\end{equation*}
Then, applying
\begin{multline*}
	\left(\la v, I \ra^{2} +  \la v_*, I_* \ra^{2}   \right)^{k/2} -  \la v, I  \ra^{k}  -  \la v_*, I_* \ra^{k}  \\ \leq 
	 2^{ {k}/2+1} \left( \left\langle v, I \right\rangle^2 \left\langle v_*, I_* \right\rangle^{k- 2} \mathds{1}_{ \left\langle v, I \right\rangle \leq \left\langle v_*, I_* \right\rangle} + \left\langle v, I \right\rangle^{k- 2} \left\langle v_*, I_* \right\rangle^2  \mathds{1}_{   \left\langle v_*, I_* \right\rangle \leq \left\langle v, I  \right\rangle} \right),
\end{multline*}
together with an upper bound on the collision kernel
\begin{equation*}
\Bf(v,v_*,I,I_*,\sigma)  \leq b(\hat{u} \cdot\sigma) \   C_\g   \left( \la v, I \ra^{\g}   + \la v_*, I_* \ra^{\g} \right),
\end{equation*}
implies, after a convenient use of the indicator function,
\begin{equation*}
	\begin{split}
	W & \leq  2^{ {k}/2+2} \, C_\g  \| b\|_{L^1}  \left( \mvI_2[f] \mvI_{k-2+\g}[g] + \mvI_2[g] \mvI_{k-2+\g}[f] \right) \\ & \leq  2^{ {k}/2+2} \, C_\g  \| b\|_{L^1}  \left( \mvI_2[f] \mvI_{k}[g] + \mvI_2[g] \mvI_{k}[f] \right),
	\end{split}
\end{equation*}
by the monotonicity of moments, which completes the proof.
\EndProof
\end{proof}

When applied to the solution $f(t,v,I)$ of the Boltzmann equation \eqref{BE frozen}, Proposition \ref{Prop weak} implies
\begin{equation}\label{sol}
\frac{\md}{\md t}\mv_{k}[f] =	\int_{\ivI} 	\Qf(f,f)(v,I)  \, \la v \ra^k \, \md v\, \md I    \leq 	 - A_{k} (\mv_{k}[f])^{1+\frac{\g}{k-2}}  + B_k,
\end{equation}
where the constants are abbreviated as $A_k = A_{k}[f,f]$ and $B_k=B_k[f,f]$.  Then, it is classical \cite{Alonso-Gamba-BAMS,Wenn} to prove generation and propagation of $v-$moments.
\begin{theorem}[Polynomial $v-$moment generation and propagation estimate] 
	Let $f$ be a solution of the Boltzmann equation \eqref{BE frozen}, having $\mvI_2$-moment finite. Define, respectively to constants in \eqref{sol},
	\begin{equation}\label{Ek}
		E_k   =   \left( \frac{B_{k}}{A_k} \right)^{\frac{k-2}{k-2+\g}}.
	\end{equation} Then the following estimates hold, for any $k>2$ and $t>0$,
\begin{align}
	& \text{1. \ (generation)} \quad 	\mv_{k}[f](t) \leq {E}_{k}   +  \left( \frac{k-2}{\g A_k } \right)^\frac{k-2}{\g} t^{-\frac{k-2}{\g}}, \qquad   \label{poly gen}\\
	&\label{poly prop} \text{2. \ (propagation)} \ \ 	\text{If $		\mv_{k}[f](0) < \infty$,   then} \ \	\mv_{k}[f](t) \leq \max \left\{ 	E_k,	\mv_{k}[f_0] \right\}. \qquad  
\end{align}
\end{theorem}
Thus,  only the partial $v-$moments are generated, while both  $v-$moments and  $I-$moments are propagated.

\section{Moment estimates for a convex combination of  pure polyatomic and frozen collisions}\label{Sec: Convex}
In this section, the aim is to combine the results of  previous Sections \ref{Sec: Frozen} and \ref{Sec: ME frozen} for frozen collisions with already established  theory in \cite{Gamba-Colic-poly, Alonso-Gamba-Colic, MC-Alonso-Pesaro} for a polyatomic gas with pure polyatomic or non-frozen  collisions \eqref{coll CL poly} which interchange the internal energy.  Let us go ahead and write directly the pure polyatomic collision operator accounting for collisions \eqref{coll CL poly},
\begin{multline}\label{coll operator general}
	Q(f,g)(v,I) = \int_{\ivI} \int_{\ipar \times [0,1]^2} \left\{ f(v',I')g(v'_*,I'_*) \left(\frac{I I_*}{I' I'_*}\right)^{\alpha}   - f(v,I)g(v_*, I_*) \right\} \\ \times   {B}(v,v_*,I,I_*,\param)  \, r^{\alpha}(1-r)^{\alpha} \, (1-R)^{2\alpha+1}  \sqrt{R}  \, \md r \, \md R \, \dstar,
\end{multline}
where $\alpha>-1$,  and collisions \eqref{coll CL poly} are  parameterized  with   $(\sigma,r,R)\in\mathbb{S}^\times[0,1]^2$,
\begin{equation}\label{coll rules}
	\begin{alignedat}{2}
		v'  &= \frac{v+v_*}{2} + \sqrt{\frac{ R  \, E}{m}} \sigma,  \qquad 	&v'_{*} &=  \frac{v+v_*}{2}  -   \sqrt{\frac{  R \,  E}{m}} \sigma,\\
		I' &=r (1-R)E,	\quad &I'_* &= (1-r)(1-R) E, 
	\end{alignedat}
\end{equation}
where $E$ is given in \eqref{assump-tilde-B}. The collision kernel $B$ is an a.e. non-negative function  assumed to take a form like   in Section \ref{Sec: assumpt coll kernel}, i.e. to be factorized 
\begin{equation}\label{assumpt B factor full}
	B(v,v_*,I,I_*,\sigma,r,R)   =  b(\hat{u} \cdot\sigma) \ \tilde{B}(v,v_*,I,I_*,r,R),
\end{equation}
with a.e. non-negative $b(\hat{u} \cdot\sigma) \in L^1(\mathbb{S}^2;\md \sigma)$ and
\begin{equation}\label{assump-tilde-B-full} 
 \tilde{b}^{lb}(r, R) \,   \left( \frac{E}{m} \right)^{\g/2} \leq  \tilde{{B}}(v,v_*,I,I_*,\rR) \leq     \tilde{b}^{ub}(r, R) \,     \left( \frac{E}{m} \right)^{\g/2}, \quad \g \in (0,2],
\end{equation}
satisfying   $	\tilde{b}^{lb}(r, R), \, \tilde{b}^{ub}(r, R)  \in L^1([0,1]^2; \, r^{\alpha}(1-r)^{\alpha} \, (1-R)^{2\alpha+1}  \sqrt{R} \, \md r \, \md R)$. 

The averaging over the set of collision parameters   $(\param)$ can be exploited  similar to $\sigma$-averaging analogue to the  Povzner Lemma \ref{Lemma Povzner}, to obtain the following result. Due to the potential discrepancy  in  $ \tilde{b}^{lb}$ and $ \tilde{b}^{ub}$ as stated in \eqref{assump-tilde-B-full}, a desired decay is guaranteed to happen only after some $k_*>2$, as explained in the upcoming lemma.
\begin{lemma}[Povzner $(\param)$-averaging, Lemma 4.3 from \cite{Gamba-Colic-poly}] There exist   an explicit non-negative constant $ \mathcal{C}_k$ decreasing in $ k\geq0$   with $\lim_{k\rightarrow\infty} \mathcal{C}_k=0$, such that
\begin{multline*}
\int_{\ipar \times [0,1]^2}	 \left(  \la v', I' \ra^{k} + \la v'_*, I'_* \ra^{k}  \right)   b(\hat{u}\cdot \sigma) \,  \tilde{b}^{ub}(r, R)  \\ \times    \,r^{\alpha}(1-r)^{\alpha} \, (1-R)^{2\alpha+1}  \sqrt{R} \,  \md r \, \md R \, \md \sigma
\leq  \mathcal{C}_k  \left(  \la v, I \ra^{2} + \la v_*, I_* \ra^{2}  \right)^{k/2}.	
\end{multline*}	
In particular, there exists $k_*>2$, depending only on the angular part $b(\hat{u}\cdot\sigma)$ and the  function $\tilde{b}^{ub}(\rR)$,  such that
\begin{equation}\label{ks povzner}
	\mathcal{C}_k <  \int_{[0,1]^2 \times \mathbb{S}^2} b(\hat{u}\cdot \sigma) \,
		\tilde{b}^{lb}(r, R)
	r^{\alpha}(1-r)^{\alpha} \, (1-R)^{2\alpha+1}  \sqrt{R} \, \md r \, \md R \, \md \sigma, \quad \text{for} \;\; k \geq k_*.
\end{equation}	
\end{lemma}	

Next, we recall moment  estimates on the collision operator \eqref{coll operator general}.
 \begin{proposition}[Lemmas 5.6 and 5.8 from \cite{Alonso-Gamba-Colic}]\label{Q est poly} Let $\g \in (0,2]$. For a suitable $f$, there exist non-negative constants $\bar{A}_k>0$, $\bar{B}_k$ and $\bar{D}_k$ such that the following estimates hold on the pure polyatomic operator \eqref{coll operator general}
 	\begin{align}
 		&\text{(a)} \    \text{ for } \ 		k \geq k_*, \quad	
 		\int_{\ivI}   	Q(f,f)(v,I) \, \la v, I \ra^k \, \md v\, \md I   	\leq
 		-    \bar{A}_{k}  \mvI_{k+\g}[f]
 		+
 		\bar{B}_k, \quad \label{L1 k large}\\
 		&\text{(b)} \  \text { for } \  k>2, \quad
 		\int_{\ivI}   	Q(f,f)(v,I) \,  \la v, I \ra^k \, \md v\, \md I     	\leq
 		\bar{D}_k   \, \mvI_k[f], \quad \label{L1 k small}
 	\end{align}
 where $k_*$ is from \eqref{ks povzner}.
 \end{proposition}

In the sequel, we consider an $\omega$-convex combination \eqref{omega coll op BE} of the pure polyatomic operator  from \eqref{coll operator general} and its frozen-counterpart   \eqref{coll operator frozen}, with collision kernels satisfying \eqref{assump-tilde-B} and \eqref{assump-tilde-B-full} with possibly different potential rates emphasized by  respectively $\g$ and $\gf$ in the subscript of the collision operators.
The corresponding Cauchy problem for the Boltzmann equation   reads \eqref{omega coll op BE}
\begin{equation}\label{BE om}
	\partial_t f(t,v,I) = \Qom(f,f), \quad f(0,v,I)=f_0(v,I).
\end{equation}

We gather Propositions \ref{Q est poly} and \ref{Prop weak} to state the following Proposition.
\begin{proposition}[Polynomial moments of the $\omega-$convex collision operator]\label{Prop weak omega}
Let $\g\in(0,2]$ be the rate of the collision kernel of the pure polyatomic operator \eqref{coll operator general}--\eqref{assumpt B factor full}--\eqref{assump-tilde-B-full} and $\gf \in[0,2]$ the rate of the frozen counterpart   \eqref{coll operator frozen}--\eqref{assumpt B factor}--\eqref{assump-tilde-B}. For suitable $f$, the following   estimate holds on the collision operator $\Qom$ defined in \eqref{omega coll op BE}, for  the convex factor $\omega>0$,
	\begin{align}
	&\text{(a)} \   \text{ for } \ 		k \geq k_*, \ \	
	\int_{\ivI}  \Qom(f,f) (v,I)  \la v, I \ra^k \, \md v\, \md I   
\leq - A_k^\omega \mvI_{k}[f]^{1+\frac{\g}{k-2}}   + {B}^\omega_k, \label{L1 est omega large k}\\
	&\text{(b)} \  \text { for } \  k>2, \ \ 
	\int_{\ivI}  \Qom(f,f) (v,I)  \la v, I \ra^k \, \md v\, \md I   
\leq   D^\omega_k   \, \mvI_k[f], \label{L1 est omega small k}
\end{align}
where the non-negative constants $\tilde{A}_k>0$, ${B}^\omega_k$ and  ${D}^\omega_k$  are   explicitly computed.
\end{proposition}
\begin{proof}
Gathering estimates \eqref{L1 k large} and \eqref{vI moment}, integration of the collision operator $ \Qom$ against $\la v, I \ra^k$ implies, for $k>k_*$,
\begin{multline*}
	\int_{\ivI}  \Qom(f,f) (v,I) \, \la v, I \ra^k \, \md v\, \md I   
	\\	\leq   	-  \omega \,  \bar{A}_{k}     \mvI_{k+\g}[f]  + (1-\omega)    D_k \mvI_2[f] \mvI_{k}[f]     + \omega \, \bar{B}_k.
\end{multline*}
Next, by moment interpolation formula \eqref{mi} and  Young's inequality,
\begin{equation*}
	\begin{split}
		(1-\omega)    D_k \mvI_2 \mvI_{k}  & \leq  (1-\omega)    D_k  \mvI_2^{\frac{\g}{k-2+\g}+1} \mvI_{k+\g}^{\frac{k-2}{k-2+\g}}  \\
		& \leq \tilde{K} + \delta \,  \mvI_{k+\g}, \quad \text{with} \ \  \tilde{K}= {( (1-\omega)    D_k )^{\frac{k-2+\g}{\g}} \mvI_2^{\frac{k-2+\g}{\g}+1} }{  \delta^{-\frac{k-2}{\g}}}.
	\end{split}
\end{equation*}
Thus, for $ \omega> 0$, choosing $\delta = \omega \bar{A}_k/2$, the statement \eqref{L1 est omega large k} follows, by denoting 
$$A_k^\omega = \frac{ \omega \,  \bar{A}_{k}}2  \mvI_2[f]^{-\frac{\g}{k-2}}, \qquad  {B}^\omega_k = \omega \, \bar{B}_k +  \tilde{K}.$$
The second inequality  \eqref{L1 est omega small k}, for $k>2$, follows by gathering \eqref{vI moment} and \eqref{L1 k small}, and by denoting $D_k^\omega =  	\omega \, \bar{D}_k  + (1-\omega)    D_k \mvI_2[f]$, which completes the proof. 
\EndProof
\end{proof}
When applied to the Boltzmann equation \eqref{BE om}, since 
\begin{equation*}
	\begin{split}
		\frac{\md}{\md t}\mvI_{k}[f] =	\int_{\ivI}  \Qom(f,f) (v,I) \, \la v, I \ra^k \, \md v\, \md I,
	\end{split}
\end{equation*}
Proposition \ref{Prop weak omega} implies the following a priori estimates on the solution. 
\begin{theorem}[Polynomial moment generation and propagation estimate] 
	Let $f$ be a solution of the Boltzmann equation \eqref{BE om}, having $\mvI_2$-moment finite. For any $k>2$, $\g\in(0,2]$ and $\gf\in[0,2]$,  define  the following constants, respectively to those  in Proposition \ref{Prop weak omega},
\begin{equation}\label{Ek omega}
	\begin{split}
		E_k^{\omega}   &=   \left( \frac{B^\omega_{k}}{A_k^\omega} \right)^{\frac{k-2}{k-2+\g}}, \qquad \mathcal{E}^\omega_k =  \mvI_2[f]^\frac{k_* - k+1}{k_*-1}   \left( {E}^\omega_{k_*+1}\right)^{\frac{k-2}{k_*-1}}  \\ 
		\tilde{\mathcal{E}}^\omega_k &=  \mathcal{E}^\omega_k   + 	\mvI_2[f]^\frac{k_* - k+1}{k_*-1} \left( \frac{(k_*-1)D_k^\omega}{\g A_k^\omega } \right)^\frac{k-2}{\g}. 
	\end{split}
\end{equation}
Then, the following estimates hold, for $t>0$,\\
	
	\noindent 1.  (generation) 
	\begin{align}
		&\text{(a)} \   \text{ for } \ 		k \geq k_*, \quad	
		\label{poly gen om}
		\mvI_{k}[f](t) \leq 	E_k^{\omega}   +  \left( \frac{k-2}{\g A_k^\omega } \right)^\frac{k-2}{\g} t^{-\frac{k-2}{\g}},  \\
		&\text{(b)} \  \text { for } \  2< k < k_*, \quad
	\mvI_{k}[f](t)	\leq \mathcal{E}^\omega_k   + 	\mvI_{2}[f]^\frac{k_* - k+1}{k_*-1}  \left( \frac{k_*-1}{\g A_k^\omega } \right)^\frac{k-2}{\g} t^{-\frac{k-2}{\g}}. \label{poly gen small k}
	\end{align}	

\noindent	2. (propagation) Moreover, if $	\mvI_{k}[f](0) =	\mvI_{k}[f_0]  < \infty$, then
	\begin{align}
		&\text{(a)} \   \text{ for } \ 		k \geq k_*, \quad		\mvI_{k}[f](t) \leq \max \left\{ 	E^\omega_k,	\mvI_{k}[f_0] \right\}, \label{poly prop large k} \\
		&\text{(b)} \  \text {  for } \  2< k < k_*, \quad 
		\mvI_{k}[f](t) \leq \max \left\{  \tilde{\mathcal{E}}^\omega_k, e \,		\mvI_{k}[f_0]  \right\}, \qquad \qquad \qquad  \ \ \ \label{poly prop small k}
	\end{align}
	where $e$ is the Euler's number.
\end{theorem}
The proof follows the same steps as for the pure polyatomic case and is derived in detail in \cite{Alonso-Gamba-Colic}, Theorem 6.2.

\begin{acknowledgement}
R. Alonso thanks TAMUQ internal funding research grant 470242-25650. M. \v{C}oli\'c thanks grants from the Ministry of Science, Technological Development and Innovation of the Republic of Serbia (Grants Nos. 451-03-137/2025-03/200125 and 451-03-136/2025-03/200125), and gratefully acknowledges the support from the Field of Excellence COLIBRI and the hospitality of the Department of Mathematics and Scientific Computing at the University of Graz, where part of this work was conducted.
\end{acknowledgement}

%
%
%

\end{document}